\documentclass[10pt,onesite,draft]{article}
\newfam\cyrfam

\usepackage{amssymb,latexsym,amscd,amsmath,amsfonts, amsthm,enumerate}

\newtheorem{theorem}{Theorem}[section]
\newtheorem{lemma}[theorem]{Lemma}
\newtheorem{proposition}[theorem]{Proposition}
\newtheorem{corollary}[theorem]{Corollary}

\newtheoremstyle{definition}
  {6pt}
  {6pt}
  {}
  {}
  {\bfseries}
  {.}
  {.5em}
  {}%

\theoremstyle{definition}
\newtheorem{definition}[theorem]{Definition}

\newtheoremstyle{remark}
  {6pt}
  {6pt}
  {}
  {}
  {\bfseries}
  {.}
  {.5em}
  {}%

\theoremstyle{remark}
\newtheorem{remark}[theorem]{Remark}

\makeatletter
\renewcommand\@makefntext[1]{%
\setlength\parindent{1em}%
\noindent
\makebox[1.8em][r]{}{#1}}
\makeatother

\begin{document}
\parskip 4pt
\large
\setlength{\baselineskip}{15 truept}
\setlength{\oddsidemargin} {0.5in}
\overfullrule=0mm
\def\bfh{\vhtimeb}
\date{}
\title{\bf \large  THE EULER-POINCARE CHARACTERISTIC\\
AND  MIXED MULTIPLICITIES}
\def\b{\vntime}
\author{
\begin{tabular}{ll}
 Duong Quoc Viet & Truong Thi Hong Thanh \\
\small duongquocviet@fmail.vnn.vn   & \small   thanhtth@hnue.edu.vn\\
\end{tabular}\\
\small Department of Mathematics, Hanoi National University of Education\\
\small 136 Xuan Thuy street, Hanoi, Vietnam\\
(Accepted by  Kyushu J. Math)}
 \date{}
\maketitle
\centerline{
\parbox[c]{11.20cm}{
\small{\bf ABSTRACT:} This paper defines  mixed  multiplicity
systems; the Euler-Poincare characteristic and the mixed
multiplicity symbol of $\mathbb{N}^d$-graded modules  with respect
to a mixed multiplicity system, and proves that the Euler-Poincare
characteristic and the mixed multiplicity symbol of any mixed
multiplicity system of the type $(k_1,\ldots,k_d)$ and the
$(k_1,\ldots,k_d)$-difference of the Hilbert polynomial are the
same. As an application, we get results for mixed
multiplicities.}}

\section{Introduction}
The mixed multiplicity is an important invariant  of Algebraic
Geometry and Commutative Algebra.  Risler-Teissier \cite{Te} in
1973 showed that each mixed multiplicity of
ideals of dimension $0$ in  a  Noetherian  local ring
is the multiplicity of an ideal generated by a superficial
sequence. Katz and Verma in 1989 \cite{KV} started the
investigation of mixed multiplicities of ideals of positive
height. For the case of arbitrary ideals, Viet \cite{Vi} in 2000
described  mixed multiplicities as the Hilbert-Samuel multiplicity
via (FC)-sequences. Trung and Verma in 2007 \cite{TV} interpreted
mixed volumes of polytopes as  mixed multiplicities of ideals.
Moreover,  by using filter-regular sequences, Manh and Viet
\cite{Vi6} in 2011 characterized  mixed multiplicities of
multi-graded modules in terms of the length  of modules.
 In past years, the
theory of mixed multiplicities has attracted much attention and
has been continually developed (see e.g. [4, 5, {11$-$15, 18$-$32]).

\enlargethispage{1cm} \footnotetext{\begin{itemize}\item[ ]This
research was in part supported by a grant from  NAFOSTED. \item[
]{\bf Mathematics Subject  Classification (2010):} Primary 13H15.
Secondary 13A02, 13E05, 13E10, 14C17. \item[ ]{\bf  Key words and
phrases:}  Mixed multiplicity, Euler-Poincare characteristic;
Koszul complex.
\end{itemize}}
Kirby and Rees \cite{KR1} in 1994 studied a kind of mixed
multiplicities of multi-graded modules  and  proved that these
mixed multiplicities  can be expressed as the Euler-Poincare
characteristic   of a certain sequence. However, how to  find
mixed multiplicity formulas, which are analogous to Serre's
formula (see e.g. \cite {Se} or \cite [Theorem 4.7.6]{BH1}) and
Auslander-Buchsbaum's formula (see e.g. \cite {AB} or \cite
[Theorem 4.7.4]{BH1}) in the Hilbert-Samuel multiplicity theory,
is not yet known. This problem became an open question of the
mixed multiplicity theory. This paper gives an  approach different
from the one in \cite{KR1}. Our approach is begun with defining
mixed multiplicity systems and related invariants, which is
similar to the one in \cite{AB, BH1}.
 We  need to choose exactly objects suitable for the goal of the
 paper and  give
mixed multiplicity formulas, which are analogous to formulas in
the Hilbert-Samuel multiplicity theory.

 Let $(A,\frak{m})$ be an Artinian local ring with  maximal ideal $\frak{m}$
 and infinite residue field $ A/\frak{m}.$
 Denote by $\mathbb{N}$ the set of all the non-negative integers. Let $d$ be a positive integer.
 Put ${\bf e}_i= (0, \ldots, \underset{(i)}{1} , \ldots, 0) \in \mathbb{N}^d$ for each $1 \leq i \leq d$ and
 $\mathrm{\bf k}!= k_1!\cdots k_d!;$\; $\mid\mathrm{\bf k}\mid = k_1+\cdots+k_d $
for any $\mathrm{\bf k}= (k_1,\ldots,k_d)\in \mathbb{N}^d.$
Moreover, set ${\bf 0} = (0,\ldots, 0) \in \mathbb{N}^d;$
${\bf 1} = (1,\ldots, 1) \in \mathbb{N}^d$ and
$\mathrm{\bf n}^\mathrm{\bf k}= n_1^{k_1}\cdots n_d^{k_d}$ for each
$\mathrm{\bf n},  \mathrm{\bf k}\in \mathbb{N}^d$ and
$\mathrm{\bf n} \ge {\bf 1}.$ Let $S=\bigoplus_{\mathrm{\bf n}\in
\mathbb{N}^d}S_{\mathrm{\bf n}}$  be a finitely generated standard $\mathbb{N}^d$-graded
algebra over $A$
(i.e., $S$ is generated over $A$
  by elements of total degree 1) and let $M=\bigoplus_{\bf n\ge \bf 0}M_{\bf n}$
  be  a finitely generated $\mathbb{N}^d$-graded $S$-module.
  For each subset ${\bf x}$ of  $S,$ we assign ${\bf x}M = 0$ if ${\bf x} = \emptyset,$ and ${\bf x}M = ({\bf x})M $ if ${\bf x} \ne \emptyset.$
  Set $S_{++}=\bigoplus_{\mathrm{\bf n}\geq \bf 1}S_{\mathrm{\bf n}}$ and $S_i= S_{{\bf e}_i}$
  for $ 1 \le i \le d.$
Denote by $\mathrm{Proj } S$ the set of the homogeneous prime
ideals of $S$ which do not contain $S_{++}$. Put
 $$\mathrm{Supp}_{++}M=\{P\in \mathrm{Proj } S\;|\;M_P\ne 0\}.$$
Assume that $S_{\bf 1}= S_{(1,\ldots,1)}$  is not contained in $\sqrt{\mathrm{Ann}_SM}$ and  $\dim \mathrm{Supp}_{++}M= s,$ then
     by \cite[Theorem 4.1]{HHRT},
    $\ell_A(M_{\mathrm{\bf n}})$ is a polynomial of degree $s$ for all large
     $\mathrm{\bf n}.$ Denote by
$P_M(\mathrm{\bf n})$ the Hilbert polynomial of the Hilbert
function $\ell_A(M_{\mathrm{\bf n}})$. The terms of total degree
$s$ in the polynomial $P_M(\mathrm{\bf n})$ have the form $
\sum_{\mid\mathrm{\bf k}\mid\;=\;s}e(M;\mathrm{\bf
k})\dfrac{\mathrm{\bf n}^\mathrm{\bf k}}{\mathrm{\bf k}!}.$
 Then $e(M;\mathrm{\bf k})$ are   non-negative integers not all zero, called the {\it  mixed multiplicity of $M$ of the type ${\bf k}$ } \cite{HHRT}.
Now for each $\mathrm{\bf k} \in \mathbb{N}^d$ such that  $\mid \mathrm{\bf k} \mid \; \ge \dim \mathrm{Supp}_{++}M,$   we put
$$E(M;\mathrm{\bf k}) =
\begin{cases} e(M;\mathrm{\bf k})\; \text{ if }\; \mid \mathrm{\bf k}\mid\; = \dim \mathrm{Supp}_{++}M, \\
0 \quad \quad \quad \;\text{ if } \; \mid \mathrm{\bf k}\mid\;  > \dim \mathrm{Supp}_{++}M.
\end{cases}$$

Next, we  define   mixed  multiplicity systems; the Euler-Poincare
characteristic and the mixed multiplicity symbol of finitely
generated $\mathbb{N}^d$-graded $S$-modules with respect to a
mixed  multiplicity system.
 \vskip 0.2cm

\noindent{\bf Definition 1.1} (Definition \ref {de3.2}). Let ${\bf
x} = x_1,\ldots, x_n$ be a sequence in $\bigcup _{j=1}^dS_j$
consisting of $m_1$ elements of $S_1,\ldots,m_d$ elements of
$S_d$. Then ${\bf x}$ is called a {\it mixed multiplicity system
of $M$ of the type $\mathrm{\bf m}= (m_1 ,\ldots,m_d)$ } if
$\dim\mathrm{Supp}_{++}\big({M}/{\bf x}M\big)\leq 0.$

In the case that ${\bf x}$ is a mixed  multiplicity system of the type  $\mathrm{\bf m}$ and
 $\mid \mathrm{\bf m}\mid = n.$
Denote by $H_{\bullet}({\bf x},M)$ the homology of the Koszul
complex of $M$ with respect to $ {\bf x} .$ Then  $ \sum_{i=0}^n
(-1)^i \ell_A[H_i({\bf x},M)_{\mathrm{\bf n}}]$ is a constant for
$\mathrm{\bf n}\gg \bf0$ (see Remark \ref{re3.1}(ii)). And denote
by $\chi({\bf x}, M)$ this constant. Then $\chi({\bf x}, M)$  is
briefly called
 the {\it Euler-Poincare characteristic of $M$ with respect to ${\bf x}.$}

Another invariant  related  to mixed  multiplicity systems is
defined as follows.

 \vskip 0.2cm
\noindent{\bf Definition 1.2} (Definition \ref{de3.4}).
 Let ${\bf x} = x_1,\ldots, x_n$ be a mixed  multiplicity system of $M.$ If $n=0,$ then
 $\ell_A(M_{\mathrm{\bf n}})= c
  \;(\mathrm{const})$ for all  $\mathrm{\bf n} \gg \bf 0$
  and we set
$\widetilde{e}({\bf x}, M) = c.$ If $n > 0,$ we set $\widetilde{e}({\bf x}, M) =
\widetilde{e}({\bf x}', M/x_1) - \widetilde{e}({\bf x}', 0_M:x_1),$ here ${\bf x}' =  x_2,\ldots, x_n.$ We call $\widetilde{e}({\bf x}, M)$ the {\it mixed multiplicity symbol of $M$ with respect to $ {\bf x} .$}

 As one might expect, first we obtain the following theorem.
\vskip 0.2cm \noindent{\bf Main Theorem } (Theorem
\ref{th3.31}).\;{\it Let $S$ be a finitely generated standard
$\mathbb{N}^d$-graded algebra over  an Artinian local ring $A$ and
let $M$ be  a finitely generated standard $\mathbb{N}^d$-graded
$S$-module. Then for any mixed multiplicity system ${\bf x}$ of
$M$ of the type $\mathrm{\bf k},$
 we have
$$\chi({\bf x}, M) = \widetilde{e}({\bf x}, M)= \triangle
^{\mathrm{\bf k}}P_M(\mathrm{\bf n}).$$ } \;\;\;\;From this
theorem, it follows that the mixed  multiplicity of the type
$\mathrm{\bf k};$
 the Euler-Poincare characteristic and the mixed multiplicity symbol of any
  mixed  multiplicity system of the type $\mathrm{\bf k}$
   are the same by the following result.
 \vskip 0.2cm \noindent{\bf Theorem 1.3} (Theorem \ref{th3.20}).\;{\it  Let  \;$\mathbf{k}\in
\mathbb{N}^d$ such that \;
 $\dim \mathrm{Supp}_{++}M \le\; \mid \mathrm{\bf k}\mid.$
Then  for any mixed  multiplicity system ${\bf x}$ of $M$ of the
type  $\mathrm{\bf k},$ we have $$E(M;\mathrm{\bf k})=\chi({\bf
x}, M) = \widetilde{e}({\bf x}, M).$$} \;\;\;\;Theorem 1.3 not
only yields interesting consequences in the case of graded modules
 (see e.g. Theorem \ref{th3.13}; Corollary \ref {co3.17};  Corollary \ref{co3.15};
  Corollary \ref{co3.12} and Corollary \ref {co3.18})
  but also gives  applications  to  mixed multiplicities of ideals
  (see Corollary \ref{co4.7}; Theorem \ref{th4.8};
  Corollary \ref{co4.11}; Corollary \ref{co4.16} and Corollary \ref{co4.10}).

 This paper is divided into four  sections.
 Section 2 is devoted to the discussion of filter-regular sequences and  mixed  multiplicity systems.
 In Section 3, we define  the Euler-Poincare characteristic;  the
mixed multiplicity symbol and obtain the Main Theorem (Theorem
\ref {th3.31}) and consequences.
 Section 4 gives applications of Section 3  to  mixed multiplicities of ideals.

\section{Filter-Regular Sequences and mixed  multiplicity systems}

This section defines mixed multiplicities; filter-regular sequences and mixed
multiplicity systems of finitely generated $\mathbb{N}^d$-graded
modules.

Assume that $S_{\bf 1} \nsubseteq \sqrt{\mathrm{Ann}_SM}$ and
$\dim \mathrm{Supp}_{++}M= s,$ then
     by \cite[Theorem 4.1]{HHRT},
    $\ell_A(M_{\mathrm{\bf n}})$ is a polynomial of degree $s$ for all large $\mathrm{\bf n}.$ Denote by
$P_M(\mathrm{\bf n})$ the Hilbert polynomial of the Hilbert
function $\ell_A(M_{\mathrm{\bf n}})$.
  The terms of total degree $s$ in the polynomial $P_M(\mathrm{\bf n})$ have the form
$ \sum_{\mid\mathrm{\bf k}\mid\;=\;s}e(M;\mathrm{\bf
k})\dfrac{\mathrm{\bf n}^\mathrm{\bf k}}{\mathrm{\bf k}!}.$
 Then $e(M;\mathrm{\bf k})$ are  non-negative integers not all zero, called the {\it  mixed multiplicity of $M$ of the type} ${\bf k}$  \cite{HHRT}.

Denote by $\triangle ^{\mathrm{\bf k}}f(\mathrm{\bf n})$ the
$\mathrm{\bf k}$-difference of the polynomial $f(\mathrm{\bf n})$
for each $\mathrm{\bf k} \in \mathbb{N}^d.$ Then we have some
following comments.
 \vskip 0.2cm

\begin{remark}\label{re2.0}\rm  If $\dim \mathrm{Supp}_{++}M= s \ge 0,$ then
$P_M(\mathrm{\bf n}) = \sum_{\mid\mathrm{\bf
k}\mid=s}e(M;\mathrm{\bf k}) \dfrac{\mathrm{\bf n}^\mathrm{\bf
k}}{\mathrm{\bf k}!} + Q_M(\mathrm{\bf n}) $ with $\deg
Q_M(\mathrm{\bf n}) < s.$  Hence  $\triangle ^{\mathrm{\bf
k}}P_M(\mathrm{\bf n}) = e(M;\mathrm{\bf k})$ for all $\mathrm{\bf
k}\in \mathbb{N}^d$ satisfying $\mid\mathrm{\bf k}\mid\;=\;s.$ Now
we assign $\dim \mathrm{Supp}_{++}M= -\infty$ to the case that
$\mathrm{Supp}_{++}M= \emptyset$ and the degree $-\infty$ to the
zero polynomial. Then
     by \cite[Theorem 4.1]{HHRT} and \cite[Proposition 2.7]{Vi6}, we always have
     $\deg P_M(\mathrm{\bf n})= \dim \mathrm{Supp}_{++}M.$
\end{remark}

The notion of mixed multiplicities $e(M;\mathrm{\bf k})$ of the
type ${\bf k}$ of a module $M$ always requires the condition $\mid
\mathrm{\bf k}\mid = \dim \mathrm{Supp}_{++}M .$ This sometimes
becomes a obstruction in describing the relationship between mixed
multiplicities and the  Euler-Poincare characteristic and in
expressing mixed multiplicity formulas. Consequently, we need the
following extension.
 \vskip 0.2cm
\begin{definition}\label{de3.7} For each $\mathrm{\bf k} \in \mathbb{N}^d$ such that  $\mid \mathrm{\bf k} \mid \; \ge \dim \mathrm{Supp}_{++}M,$   we put
$$E(M;\mathrm{\bf k}) =
\begin{cases} e(M;\mathrm{\bf k})\; \text{ if } \mid \mathrm{\bf k}\mid\; = \dim \mathrm{Supp}_{++}M, \\\\
0 \quad \quad \quad \;\text{ if } \mid \mathrm{\bf k}\mid\;  >
\dim \mathrm{Supp}_{++}M.
\end{cases}$$
\end{definition}

We turn now to filter-regular sequences of multi-graded modules.
The notion of filter-regular sequences was introduced by Stuckrad
and Vogel in \cite{SV}(see \cite{BS}). The theory of
filter-regular sequences became an important tool to study some
classes of singular rings and has been continually developed (see
e.g. \cite{BS, Hy,  Tr2, {Vi6}, VT1}).

 \vskip 0.2cm
\begin{definition}\label{de2.1} A homogeneous element $a$ of $S$ is called an {\it $S_{++}$-filter-regular element with respect
to  $M$}  if  $(0_M:a)_{\mathrm{\bf n}}=0$ for  all large
$\mathrm{\bf n}.$ A homogeneous sequence $x_1,\ldots, x_t$ in $S$
is called an {\it $S_{++}$-filter-regular sequence with respect to
$M$} if
 $x_i$ is an $S_{++}$-filter-regular element with respect to $M/(x_1,\ldots, x_{i-1})M$ for all $i = 1,\ldots, t.$
  \end{definition}

 \begin{remark}\label{re2.2} \rm We have the following comments for filter-regular
sequences of multi-graded modules.
\begin{itemize}
 \item[$\mathrm{(i)}$] By \cite [ Note (i)] {Vi6},  a homogeneous element
  $a\in S$ is  an  $S_{++}$-filter-regular element with respect to $M$ if and only if
  $0_M:a\subseteq 0_M:S_{++}^{\infty}.$  Moreover, for each
  ${\bf k} = (k_1,\ldots,k_d) \in \mathbb{N}^d,$
  there exists an $S_{++}$-filter-regular sequence ${\bf x}$
  in $\bigcup_{i=1}^dS_i$ with respect to  $M$ consisting of
  $k_1$ elements of $S_1,\ldots,k_d$ elements of
   $S_d$ by \cite [Proposition 2.2 and Note (ii)] {Vi6}.
   In this case, ${\bf x}$ is called an $S_{++}$-filter-regular sequence
   of {\it the type} ${\bf k}.$
 \item[$\mathrm{(ii)}$] Let  $a \in S_i.$
  Since the following  exact sequence
$$0\longrightarrow (0_M: a)_{\mathrm{\bf n}-{\bf e}_i} \longrightarrow M_{\mathrm{\bf n}-{\bf e}_i}
\xrightarrow{\;\;a\;\;} M_{\mathrm{\bf n}}\longrightarrow
\dfrac{M_{\mathrm{\bf n}}} {aM_{\mathrm{\bf n}-{\bf
e}_i}}\longrightarrow 0,$$ it  follows that $\triangle ^{{\bf
e}_i}P_M(\mathrm{\bf n})= P_{M/aM}(\mathrm{\bf n})- P_{(0_M:
a)}(\mathrm{\bf n}-{\bf e}_i).$ Note that $$\deg \triangle^{{\bf
e}_i}P_M(\mathrm{\bf n}) \leq \deg P_M(\mathrm{\bf n})-1 =
\dim\mathrm{Supp}_{++}M - 1.$$ Hence $\deg P_{M/aM}(\mathrm{\bf
n}) \leq \dim \mathrm{Supp}_{++}M - 1$ if and only if $$\deg
P_{(0_M: \;a)}(\mathrm{\bf n}) \leq \dim\mathrm{Supp}_{++}M - 1.$$
So $\dim\mathrm{Supp}_{++}\bigg(\dfrac{M}{aM}\bigg)\leq
\dim\mathrm{Supp}_{++}M - 1$ if and only if $$
\dim\mathrm{Supp}_{++}(0_M: a)\leq \dim\mathrm{Supp}_{++}M - 1.$$
 If $a$ is an
$S_{++}$-filter-regular element, then $P_{(0_M:a)}(\mathrm{\bf
n})=0.$ In this case, we have $\dim\mathrm{Supp}_{++}(0_M: a) =
-\infty$ and $\triangle ^{{\bf e}_i}P_M(\mathrm{\bf n}) =
P_{M/aM}(\mathrm{\bf n}).$ Therefore  $
\dim\mathrm{Supp}_{++}\bigg(\dfrac{M}{aM}\bigg)\leq
\dim\mathrm{Supp}_{++}M - 1.$  And if there exists an
$e(M;\mathrm{\bf k}) \ne 0$ with $k_i > 0,$  then $
\dim\mathrm{Supp}_{++}\bigg(\dfrac{M}{aM}\bigg) =
\dim\mathrm{Supp}_{++}M - 1$ by \cite [Proposition 3.3 (i)]{Vi6}.

 \item[$\mathrm{(iii)}$]
A homogeneous element $x\in S$ is  an $S_{++}$-filter-regular
element with respect to  $M$ if $x\notin P$ \;\;for any $P\in
\mathrm{Ass}_SM$ not containing   $S_{++}.$\;\;That \;\; means
$x\notin\bigcup_{S_{++}\nsubseteq P,\; P\in \mathrm{Ass}_SM}P$ by
\cite[Definition 2.1 and Proposition 2.5]{Vi6}.
\end{itemize}
\end{remark}

Now, we would like to give the following concept.

 \vskip 0.2cm
\begin{definition}\label{de3.2}
Let ${\bf x} = x_1,\ldots, x_n$ be a sequence of elements in
$\bigcup _{j=1}^dS_j$ consisting of $m_1$ elements of
$S_1,\ldots,m_d$ elements of $S_d$. Then ${\bf x}$ is called a
{\it mixed  multiplicity system of $M$ of the type  $\mathrm{\bf
m}= (m_1 ,\ldots,m_d)$ } if
  $\dim\mathrm{Supp}_{++}\bigg(\dfrac{M}{{\bf x}M}\bigg)\leq 0.$
\end{definition}

Remember that in the case of modules over local rings, one gave
the multiplicity systems that generate the ideals of definition.
The results of this paper will show the usefulness of mixed
multiplicity systems.

 \vskip 0.2cm
\begin{remark}\label{re 3.0}\rm
By Remark \ref{re2.2}(i), for each ${\bf k}\in \mathbb{N}^d$,
      there exists an  $S_{++}$-filter-regular sequence ${\bf x}= x_1,\ldots, x_s$ of
the type ${\bf k}.$ Then for any $ 0 \le i \le s,$ we have
$\dim\mathrm{Supp}_{++}\bigg(\dfrac{M}{(x_1,\ldots,
x_i)M}\bigg)\leq \dim\mathrm{Supp}_{++}M - i$ by Remark
\ref{re2.2}(ii).  Now we choose ${\bf k}$ such that $\mid {\bf
k}\mid \;\ge \dim\mathrm{Supp}_{++}M.$ Then by Remark
\ref{re2.2}(ii), we get
$\dim\mathrm{Supp}_{++}\bigg(\dfrac{M}{{\bf x}M}\bigg)\leq
\dim\mathrm{Supp}_{++}M - \mid {\bf k} \mid\; \le 0.$ Hence ${\bf
x}$ is a  mixed  multiplicity system of $M$. So for any ${\bf k}
\in \mathbb{N}^d$ with $\mid {\bf k} \mid \; \ge
\dim\mathrm{Supp}_{++}M,$ there exists a  mixed  multiplicity
system $x_1,\ldots, x_s$ of $M$ of the type  $\mathrm{\bf k}$ such
that for any  $ 0 \le i \le s,$
$\dim\mathrm{Supp}_{++}\bigg(\dfrac{M}{(x_1,\ldots,
x_i)M}\bigg)\leq \dim\mathrm{Supp}_{++}M - i.$   \end{remark}

We end this section with the following property of mixed multiplicity systems.

 \vskip 0.2cm
\begin{lemma}\label{lem 3.30} Let $0\longrightarrow M' \longrightarrow M\longrightarrow M"\longrightarrow 0$ be an exact sequence
of $\mathbb{N}^d$-graded $S$-modules. And let ${\bf x}$ be  a
sequence of elements in $\bigcup_{i=1}^d S_i.$   Then ${\bf x}$ is
a mixed  multiplicity system of $M$ if and only if ${\bf x}$ is a
mixed  multiplicity system of both $M'$ and $M".$
\end{lemma}

\begin{proof} Note that $\sqrt{\mathrm{Ann}(F/{\frak c}F)}= \sqrt{\mathrm{Ann}F + {\frak c}}$ for each ideal $\frak c$ and each module $F.$
 Hence since $\mathrm{Ann}_SM \subseteq \mathrm{Ann}_SM',$  it follows that  $$\dim\mathrm{Supp}_{++}\bigg(\dfrac{M'}{{\bf x}M'}\bigg)\leq \dim\mathrm{Supp}_{++}\bigg(\dfrac{M}{{\bf x}M}\bigg).$$
On the other hand from the exactness of
$$ M'/{\bf x}M' \xrightarrow{\;g\;} M/{\bf x}M\rightarrow M"/{\bf x}M"\rightarrow 0,$$
it implies that  $\dim\mathrm{Supp}_{++}\bigg(\dfrac{M'}{{\bf
x}M'}\bigg)\geq \dim\mathrm{Supp}_{++} \mathrm{Im}\; g $ and
$$\dim\mathrm{Supp}_{++}\bigg(\dfrac{M}{{\bf x}M}\bigg) = \max\bigg\{\dim\mathrm{Supp}_{++} \mathrm{Im}\; g, \; \dim\mathrm{Supp}_{++}\bigg(\dfrac{M"}{{\bf x}M"}\bigg)\bigg\}.$$
From the above facts, we immediately get that
$\dim\mathrm{Supp}_{++}\bigg(\dfrac{M}{{\bf x}M}\bigg) \le 0$ if
and only if $\dim\mathrm{Supp}_{++}\bigg(\dfrac{M'}{{\bf
x}M'}\bigg)\leq 0$ and
$\dim\mathrm{Supp}_{++}\bigg(\dfrac{M"}{{\bf x}M"}\bigg)\leq 0.$
Thus, ${\bf x}$ is a mixed  multiplicity system of $M$ if and only
if ${\bf x}$ is a mixed  multiplicity system of both $M'$ and
$M".$
\end{proof}

\section{The Euler-Poincare Characteristic and  Multiplicities}

In this section, we  define   the Euler-Poincare characteristic
and the mixed multiplicity symbol of finitely generated standard
$\mathbb{N}^d$-graded $S$-modules with respect to a  mixed
multiplicity system. Next we prove that the Euler-Poincare
characteristic and the mixed multiplicity symbol of any  mixed
multiplicity system of the type ${\bf k} = (k_1,\ldots,k_d) \in
\mathbb{N}^d$ and the ${\bf k}$-difference of the Hilbert
polynomial are the same, and give interesting consequences for
mixed multiplicities of $\mathbb{N}^d$-graded modules.

Let ${\bf x}$ be a system  of $n$ homogeneous elements in $S.$ Then by \cite[Remark 1.6.15]{BH1}, we may consider the following Koszul complex $K_{\bullet}({\bf x},M)$ of $M$ with respect to ${\bf x}$   as an  $\mathbb{N}^d$-graded complex with a differential of degree ${\bf 0}:$
$$ 0\longrightarrow K_n({\bf x},M)\longrightarrow K_{n-1}({\bf x},M)\longrightarrow \cdots\longrightarrow K_1({\bf x},M)\longrightarrow K_0({\bf x},M) \longrightarrow 0.$$
Denote by $H_{\bullet}({\bf x},M)$ the homology of the Koszul complex of $M$ with respect to
$ {\bf x} .$   We get the following sequence of the homology modules
$$H_{\bullet}({\bf x},M):\ldots H_0({\bf x},M), H_1({\bf x},M),\ldots,H_n({\bf x},M)\ldots.$$

The  Koszul complex theory became an important tool to study
several different theories of Commutative Algebra and Algebraic Geometry.
 \vskip 0.2cm

\begin{remark}\label{re3.1}\rm
Note that $K_i({\bf x},M)$ and $H_i({\bf x},M)$ are finitely generated $\mathbb{N}^d$-graded $S$-modules. And  ${\bf x}H_i({\bf x},M) = 0$ for all $0 \leq i \leq n$ (see \cite[Proposition 1.6.5(b)]{BH1}).  Moreover,  we have the following notes.
\begin{itemize}
\item[$\mathrm{(i)}$]
 Since $\sqrt{\mathrm{Ann}_S(M/{\bf x}M)}= \sqrt{\mathrm{Ann}_SM + ({\bf x})}$ and
  $\mathrm{Ann}_SM + ({\bf x}) \subseteq \mathrm{Ann}_SH_i({\bf x},M)$ for $0 \leq i \leq n,$ it follows that
    $\mathrm{Ann}_S\bigg(\dfrac{M}{{\bf x}M}\bigg)
 \subseteq \sqrt{ \mathrm{Ann}_SH_i({\bf x},M)}$ for $0 \leq i \leq n.$ Consequently,
 if $\dim\mathrm{Supp}_{++}\bigg(\dfrac{M}{{\bf x}M}\bigg)\leq 0,$ then
$\dim\mathrm{Supp}_{++}H_i({\bf x},M)\leq 0$ for all $0 \leq i
\leq n.$ \item[$\mathrm{(ii)}$] Let ${\bf x}= x_1,\ldots, x_n$ be
a mixed  multiplicity system of $M$ of the type ${\bf m}.$
 Then  for each $0 \leq i \leq n$ we have
$\dim\mathrm{Supp}_{++}H_i({\bf x},M)\leq 0$ by (i). Therefore  by
Remark \ref{re2.0}, we get
 $\ell_A[H_i({\bf x},M)_{\mathrm{\bf n}}] = a_i (\mathrm{const})$
 for all  $\mathrm{\bf n}\gg \bf 0$ and for each $0 \leq i \leq n.$
 Hence there exists $\mathrm{\bf v}\in \mathbb{N}^d$ such that for each $0 \leq i \leq n,$ $\ell_A[H_i({\bf x},M)_{\mathrm{\bf n}}] = a_i$
for all $\mathrm{\bf n} \geq \mathrm{\bf v}.$ So  we get the
constant: $\chi({\bf x}, M) = \sum_{i=0}^n (-1)^i \ell_A[H_i({\bf
x},M)_{\mathrm{\bf n}}]$ for all $\mathrm{\bf n}\geq \mathrm{\bf
v}.$ In this case, $\chi({\bf x}, M)$ is briefly called the {\it
Euler-Poincare characteristic of $M$ with respect to $ {\bf x}.$}
\end{itemize}
\end{remark}

Some  basic properties of the Euler-Poincare characteristic with respect to mixed  multiplicity systems are stated by the following lemma.
 \vskip 0.2cm
\begin{lemma}\label{lem 3.4} Let ${\bf x} = x_1,\ldots, x_n$ be  a mixed  multiplicity system of $M.$  Then the following statements hold.
\begin{itemize}
\item[$\mathrm{(i)}$] $\chi({\bf x},\underline{\;}\; )$ is additive
on short exact sequences, i.e., if
$$0\longrightarrow M' \longrightarrow M\longrightarrow M"\longrightarrow 0$$
is a short exact sequence, then $\chi({\bf x},M)= \chi({\bf
x},M')+\chi({\bf x},M").$ \item[$\mathrm{(ii)}$] If $x_1 \in
\sqrt{\mathrm{Ann}_SM},$ then $\chi({\bf x},M)= 0.$
\item[$\mathrm{(iii)}$] If $x_1$ is $M$-regular and ${\bf x}' =
x_2,\ldots, x_n,$  then $\chi({\bf x},M)= \chi({\bf x}', M/x_1M).$
\end{itemize}
\end{lemma}
\begin{proof} Without loss of generality, we may assume that  $\mathbb{N}^d$-graded $S$-homomorphisms in this proof are $\mathbb{N}^d$-graded $S$-homomorphisms of degree $\bf 0$.

The proof of (i): Since the Koszul complex is an exact functor, the exact sequence
of $\mathbb{N}^d$-graded $S$-modules
$$0\longrightarrow M' \longrightarrow M \longrightarrow M"\longrightarrow 0$$
yields a long exact  sequence
$$\cdots\longrightarrow H_i({\bf x},M') \longrightarrow H_i({\bf x},M)\longrightarrow
H_i({\bf x},M")\longrightarrow H_{i-1}({\bf x},M')\longrightarrow
\cdots$$ of $\mathbb{N}^d$-graded homology $S$-modules.
 Hence for each $\mathrm{\bf n}\in \mathbb{N}^d$ we have an exact  sequence of $A$-modules
$$\cdots\longrightarrow H_i({\bf x},M')_{\mathrm{\bf n}} \longrightarrow
H_i({\bf x},M)_{\mathrm{\bf n}}\longrightarrow H_i({\bf x},M")_{\mathrm{\bf n}}\longrightarrow
H_{i-1}({\bf x},M')_{\mathrm{\bf n}}\longrightarrow \cdots.$$
Recall that ${\bf x}$ is also a mixed  multiplicity system of both
$M'$ and $M"$ by Lemma \ref{lem 3.30}. Hence by the additivity of
length, the alternating sum of the lengths of the modules in this
exact sequence is zero. This fact follows (i).

The proof of (ii): First we consider the case that $x_1 \in \mathrm{Ann}_SM.$ Recall that ${\bf x}' =  x_2,\ldots, x_n$. By \cite[Corollary 1.6.13(a)]{BH1}, we have the following exact sequence
$$\cdots\xrightarrow{\pm x_1} H_i({\bf x}',M) \longrightarrow H_{i}({\bf x},M)\longrightarrow H_{i-1}({\bf x}',M)\xrightarrow{\pm x_1} H_{i-1}({\bf x}',M)\longrightarrow \cdots.\eqno (3.1)$$
Since $x_1M = 0,$ ${\bf x}'$ is also a mixed  multiplicity system of $M.$ Moreover,   $x_1$ annihilates $H_{i}({\bf x}',M)$ for all $i$. Hence the exact sequence (3.1) breaks up into exact sequences
$$0\longrightarrow H_i({\bf x}',M) \longrightarrow H_{i}({\bf x},M)\longrightarrow H_{i-1}({\bf x}',M)\longrightarrow 0$$
for all $i.$ Therefore $$\ell_A[H_i({\bf x},M)_{\mathrm{\bf n}}] = \ell_A[H_i({\bf x}',M)_{\mathrm{\bf n}}] + \ell_A[H_{i-1}({\bf x}',M)_{\mathrm{\bf n}}]$$
for all $i$ and for all large  ${\mathrm{\bf n}}$. Note  that
$H_{n + 1}({\bf x},M) = H_{-1}({\bf x},M) = 0,$ so
$$
 \sum_{i=0}^{n+1} (-1)^i \ell_A[H_i({\bf x},M)_{\mathrm{\bf n}}]
= \sum_{i=0}^{n+1} (-1)^i \bigg[\ell_A[H_i({\bf x}',M)_{\mathrm{\bf n}}] + \ell_A[H_{i-1}({\bf x}',M)_{\mathrm{\bf n}}]\bigg] = 0 \eqno (3.2)
$$
for all large ${\bf n}.$ Since  $\ell_A[H_{n+1}({\bf
x},M)_{\mathrm{\bf n}}] = 0$ and $$\chi({\bf x}, M) =
\sum_{i=0}^{n} (-1)^i \ell_A[H_i({\bf x},M)_{\mathrm{\bf n}}]$$
for all ${\bf n} \gg \bf 0,$ $\chi({\bf x},M)= 0$ by (3.2). So if
$x_1M = 0$, then $\chi({\bf x},M)= 0.$ From this it follows that
$\chi({\bf x},M/aM)= 0$ for all $a \in {\bf x}.$

We turn now to the case that $x_1 \in \sqrt{\mathrm{Ann}_SM}.$
Then there exists $u > 0$ such that
$x_1^u M = 0_M.$
On the other hand since $\chi({\bf x},M/x_1M)= 0$ and the exact sequence
$$0\longrightarrow x_1M \longrightarrow M \longrightarrow M/x_1M\longrightarrow 0,$$ it follows by (i) that
$\chi({\bf x},M)= \chi({\bf x},x_1M).$ Therefore, $\chi({\bf x},M)= \chi({\bf x},x_1^jM)$ for all $j \ge 1.$ Consequently, $\chi({\bf x},M)= \chi({\bf x},x_1^uM) = \chi({\bf x}, 0_M) = 0.$
 We get (ii).

The proof of (iii): Since  $x_1$ is $M$-regular, we have $$H_i({\bf x},M) \cong H_i({\bf x}',M/x_1M)$$ by \cite[Corollary 1.6.13(b)]{BH1}. Thus
$$\ell_A[H_i({\bf x},M)_{\mathrm{\bf n}}] = \ell_A[H_i({\bf x}',M/x_1M)_{\mathrm{\bf n}}]$$
for all large ${\mathrm{\bf n}}$ and for all $i$. From this it follows (iii).
\end{proof}

Using  Lemma \ref{lem 3.4},  we prove the following recursion formula for the Euler-Poincare characteristic of a graded module with respect to a mixed  multiplicity system.

 \vskip 0.2cm
\begin{lemma}\label{th3.6}
 Let ${\bf x} = x_1,\ldots, x_n$ be a mixed  multiplicity system of the module $M.$  Set ${\bf x}' =  x_2,\ldots, x_n.$ Then  we have
$$\chi({\bf x}, M) =\chi({\bf x}', M/x_1M)- \chi({\bf x}', 0_M: x_1).$$
\end{lemma}
\begin{proof} Consider the following cases.

If  $x_1 \in \sqrt{\mathrm{Ann}_SM},$ then $\chi({\bf x}, M) = 0$
by Lemma \ref{lem 3.4}(ii). In this case,
$$\mathrm{Supp}_{++}(M/{\bf x}M) = \mathrm{Supp}_{++}(M/{\bf x}'M).$$ Hence
${\bf x}'$ is also a mixed multiplicity system of $M.$ Therefore
since the exact sequence
$$0\longrightarrow (0_M:x_1) \longrightarrow M \xrightarrow{\;\;x_1\;\;} M\longrightarrow M/x_1M\longrightarrow 0,$$
we have $\chi({\bf x}', M/x_1M)- \chi({\bf x}', 0_M: x_1) = 0$  by Lemma \ref{lem 3.4}(i).
So $$\chi({\bf x}, M) =\chi({\bf x}', M/x_1M)- \chi({\bf x}', 0_M: x_1).$$

In the case that $x_1 \notin \sqrt{\mathrm{Ann}_SM}$ and set $ D =
0_M: x_1^{\infty},$ then $x_1$ is $(M/D)$-regular.  Since $ x_1M
\bigcap D = x_1D,$ we have
 the exact sequence
 $$0\longrightarrow \frac{D}{x_1D}\longrightarrow \frac{M}{x_1M} \longrightarrow \frac{M}{x_1M +D} \longrightarrow 0.$$
Hence $$\chi({\bf x}', \frac{M}{x_1M+D}) = \chi({\bf x}',
\frac{M}{x_1M})- \chi({\bf x}', \frac{D}{x_1D})$$ by Lemma
\ref{lem 3.4}(i). Since $M$ is Noetherian, $D = 0_M: x_1^v$ for
some $v > 0.$ Hence $x_1  \in \sqrt{\mathrm{Ann}_SD}.$ So ${\bf
x}'$ is also a mixed  multiplicity system of $D.$ On the other
hand $\chi({\bf x}', \frac{D}{x_1D}) = \chi({\bf x}', 0_M:x_1)$
since the exact sequence
$$0\longrightarrow (0_M:x_1) \longrightarrow D \xrightarrow{\;\;x_1\;\;} D\longrightarrow \frac{D}{x_1D}\longrightarrow 0.$$
Thus
$$\chi({\bf x}', \frac{M}{x_1M+D}) = \chi({\bf x}', \frac{M}{x_1M})- \chi({\bf x}', 0_M:x_1).\eqno(3.3)$$
Recall that $x_1$ is  $(M/D)$-regular.  So
 $\chi({\bf x}', \frac{M}{x_1M+D}) = \chi({\bf x}, \frac{M}{D})$ by
Lemma \ref{lem 3.4}(iii). Now since $\chi({\bf x}, \frac{M}{D})=
\chi({\bf x}, M)-\chi({\bf x}, D)$ by Lemma \ref{lem 3.4}(i), we
get
$$\chi({\bf x}', \frac{M}{x_1M+D}) = \chi({\bf x}, M)-\chi({\bf x}, D).$$
Note that $x_1  \in \sqrt{\mathrm{Ann}_SD},$ hence
$\chi({\bf x}, D) = 0$ by Lemma \ref{lem 3.4}(ii). So
$$\chi({\bf x}', \frac{M}{x_1M+D}) = \chi({\bf x}, M).\eqno(3.4)$$
By (3.3) and (3.4) we obtain  $$\chi({\bf x}, M) =\chi({\bf x}', M/x_1M)- \chi({\bf x}', 0_M: x_1).$$
The lemma is proved.\end{proof}

Next we construct an invariant that is called the mixed multiplicity symbol with respect to  a mixed  multiplicity system.

 \vskip 0.2cm
\begin{definition}\label{de3.4} Let ${\bf x} = x_1,\ldots, x_n$ be a mixed  multiplicity
system of $M.$ If $n=0,$ then $\ell_A(M_{\mathrm{\bf n}})= c
\;(\mathrm{const})$ for all  $\mathrm{\bf n} \gg \bf 0$ and we set
$\widetilde{e}({\bf x}, M) = \widetilde{e}(\emptyset, M)= c.$ If
$n > 0,$ we set $$\widetilde{e}({\bf x}, M) = \widetilde{e}({\bf
x}', M/x_1M) - \widetilde{e}({\bf x}', 0_M:x_1),$$ here ${\bf x}'
=  x_2,\ldots, x_n.$ We call $\widetilde{e}({\bf x}, M)$ the {\it
mixed multiplicity symbol of $M$ with respect to $ {\bf x} .$}

\end{definition}

Then the relationship between the Euler-Poincare characteristic  and
the mixed multiplicity symbol with respect to a mixed  multiplicity system of $M$ is given by the following proposition.

\begin{proposition}\label{th3.5}  Let ${\bf x}$ be a mixed  multiplicity system of $M.$ Then we have $$ \chi({\bf x}, M) = \widetilde{e}({\bf x}, M).$$
\end{proposition}
\begin{proof} Note that if the length of ${\bf x}$ is equal to $0$, then we have
$$\chi({\bf x}, M) = \ell_A[H_0({\bf x},M)_{\mathrm{\bf n}}] = \ell_A(M_{\mathrm{\bf n}}) = \widetilde{e}({\bf x}, M)$$
for all large ${\mathrm{\bf n}}$. Consequently  the assertion follows from Lemma \ref{th3.6} and the definition of the mixed multiplicity symbol of $M$ with respect to ${\bf x}$ (see Definition \ref{de3.4}).
\end{proof}

The Euler-Poincare characteristic of a mixed  multiplicity system
of $M$ of the type  $\mathrm{\bf k}$ and the ${\bf k}$-difference
of the Hilbert polynomial $P_M(\mathrm{\bf n})$ are directly
linked by the following proposition.

 \vskip 0.2cm
\begin{proposition}\label{th3.11}  Let $P_M(\mathrm{\bf n})$ be the Hilbert polynomial of the Hilbert
function $\ell_A(M_{\mathrm{\bf n}})$. Then for any mixed
multiplicity system ${\bf x}$ of $M$ of the type  $\mathbf{k}\in
\mathbb{N}^d,$
 we have $ \chi({\bf x}, M) =  \triangle ^{\mathrm{\bf k}}P_M(\mathrm{\bf n}).$
\end{proposition}

\begin{proof}  Set $\mid \mathrm{\bf
k}\mid = s.$ Denote by $r$ the number of all non-zero elements in
$ {\bf x} .$
  We  will  prove that
  $\chi({\bf x}, M) = \triangle ^{\mathrm{\bf k}}P_M(\mathrm{\bf n})$
  by induction on $r.$

Consider the case that $r = 0.$ Then $\dim\mathrm{Supp}_{++}M\leq
0,$ so $\deg P_M(\mathrm{\bf n}) \leq  0.$ Now if $s = 0,$ then we
have  $\chi({\bf x}, M) = \ell_A[H_0({\bf x},M)_{\mathrm{\bf n}}]=
\ell_A(M_{\mathrm{\bf n}})$ for all $ \mathrm{\bf n} \gg \bf  0$
and
$$\triangle ^{\mathrm{\bf k}}P_M(\mathrm{\bf n})=
\triangle ^{\mathrm{\bf 0}}P_M(\mathrm{\bf n})=
\ell_A(M_{\mathrm{\bf n}})$$ for all $ \mathrm{\bf n} \gg \bf 0.$
Hence $\chi({\bf x}, M) = \triangle ^{\mathrm{\bf
k}}P_M(\mathrm{\bf n}).$ If $s > 0,$ then $\chi({\bf x}, M)= 0$ by
Lemma \ref{lem 3.4}(ii).
 Since $\mid \mathrm{\bf k}\mid = s > 0$ and $\deg P_M(\mathrm{\bf n}) \leq 0,$
 it follows that $\triangle ^{\mathrm{\bf k}}P_M(\mathrm{\bf n}) = 0.$
 Thus $$\chi({\bf x}, M) = \triangle ^{\mathrm{\bf k}}P_M(\mathrm{\bf n}).$$ Therefore,
  if $r = 0,$ then
$\chi({\bf x}, M) = \triangle ^{\mathrm{\bf k}}P_M(\mathrm{\bf
n}).$

Next assume  that $r>0;$  ${\bf x} = x_1,\ldots, x_s$ and $x_1 \ne
0.$ Set ${\bf x}' =  x_2,\ldots, x_s.$  Let $$0 = M_0 \subseteq
M_1 \subseteq M_2 \subseteq\cdots \subseteq M_u=M$$ be a prime
filtration of $M,$ i.e., $M_{j+1}/M_j \cong S/P_j$ where  $P_j$ is
a homogeneous prime ideal for all $0 \le j \le u-1$ and
$\{P_0,P_1,\ldots,P_{u-1}\} \subseteq \text{Supp}M.$ Then we get
$$P_M(\mathrm{\bf n}) = \sum_{j = 0}^{u-1}P_{S/P_j}(\mathrm{\bf
n}).$$ By Lemma \ref{lem 3.30}, ${\bf x}$ is a  mixed multiplicity
system of  $S/P_j$ for each $0 \le j \le u-1.$ And by Lemma
\ref{lem 3.4}(i), it follows that $\chi({\bf x},M) = \sum_{j =
0}^{u-1}\chi({\bf x}, S/P_j).$

Now if $x_1 \in P_j$ and denote by $\overline {\bf x}$ the image
of ${\bf x}$ in $S/P_j$   and consider $\chi(\overline {\bf x},
S/P_j)$ as the  Euler-Poincare characteristic of the
$S/P_j$-module $S/P_j$ with respect to $\overline {\bf x}$, then
$\chi(\overline {\bf x}, S/P_j)= \chi({\bf x}, S/P_j)$ and
$\overline x_1 = 0$ in $S/P_j.$ Hence if we call $v$ the number of
all non-zero elements in $\overline {\bf x}$, then $v < r.$  Then
  by the inductive assumption  we have
$\triangle ^{\mathrm{\bf k}}P_{S/P_j}(\mathrm{\bf n}) =
\chi({\overline {\bf x}}, S/P_j).$ Therefore $\triangle
^{\mathrm{\bf k}}P_{S/P_j}(\mathrm{\bf n})= \chi({\bf x}, S/P_j).$

In the case that $x_1 \notin P_j,$ then $x_1$ is $S/P_j$-regular
of $S$-module $S/P_j.$ Now assume that $x_1 \in S_i.$  Then by
\cite[Remark 2.6]{Vi6}, it follows that $\triangle ^{{\bf
e}_i}P_{S/P_j}(\mathrm{\bf n}) = P_{S/(x_1,P_j)}(\mathrm{\bf n}).$
 Since the number of all non-zero
elements in ${\bf x}'$ is less than $r$ and ${\bf x}'$ is a  mixed
multiplicity system of the type $ {\bf k}- {\bf e}_i \in
\mathbb{N}^d$ of $S/(x_1,P_j)$, we get $$\triangle ^{{\bf k}- {\bf
e}_i}P_{S/(x_1,P_j)}(\mathrm{\bf n})= \chi({\bf x}',
S/(x_1,P_j))$$ by the inductive assumption.  Note that
$$\triangle ^{{\bf k}- {\bf e}_i}P_{S/(x_1,P_j)}(\mathrm{\bf n})=
\triangle ^{{\bf k}- {\bf e}_i}[\triangle ^{{\bf
e}_i}P_{S/P_j}(\mathrm{\bf n})]= \triangle ^{\mathrm{\bf
k}}P_{S/P_j}(\mathrm{\bf n}).$$ Since $x_1$ is $S/P_j$-regular of
$S$-module $S/P_j,$ we have $\chi({\bf x}', S/(x_1,P_j))=\chi({\bf
x}, S/P_j)$ by  Lemma \ref{lem 3.4}(iii).  So $\triangle
^{\mathrm{\bf k}}P_{S/P_j}(\mathrm{\bf n}) = \chi({\bf x}, S/P_j)$
for each $0 \le j \le u-1.$
  Consequently
$$\triangle ^{\mathrm{\bf k}}P_M(\mathrm{\bf n}) = \sum_{j = 0}^{u-1}\triangle ^{\mathrm{\bf k}}P_{S/P_j}(\mathrm{\bf n})= \sum_{j = 0}^{u-1}\chi({\bf x}, S/P_j)= \chi({\bf x},M).$$
Thus $\chi({\bf x}, M) = \triangle ^{\mathrm{\bf
k}}P_M(\mathrm{\bf n}).$ Induction is complete.
\end{proof}
  Proposition \ref{th3.5} and  Proposition \ref{th3.11} yield:
 \vskip 0.2cm
\begin{theorem}\label{th3.31} Let $S$ be a finitely generated standard $\mathbb{N}^d$-graded
algebra over  an Artinian local ring $A$ and let $M$ be  a
finitely generated standard $\mathbb{N}^d$-graded $S$-module. Then
for any mixed  multiplicity system ${\bf x}$ of $M$ of the type
$\mathrm{\bf k},$
 we have
$$\chi({\bf x}, M) = \widetilde{e}({\bf x}, M)= \triangle
^{\mathrm{\bf k}}P_M(\mathrm{\bf n}).$$
\end{theorem}

\begin{remark}\label{re 3.3}\rm  From Theorem \ref{th3.31}, it follows that if $\mid \mathrm{\bf k}\mid \; > \dim \mathrm{Supp}_{++}M,$
then $\chi({\bf x}, M) = \widetilde{e}({\bf x}, M) = 0$ since
$\triangle ^{\mathrm{\bf k}}P_M(\mathrm{\bf n})=0.$ Moreover, the
mixed multiplicity
 symbol $\widetilde{e}({\bf x}, M)$
 does not depend on the order of the elements of ${\bf x}$ and the
 Euler-Poincare characteristic $\chi({\bf x}, M)$   depends only on the type of $ {\bf x} .$
\end{remark} Now if $\mid\mathrm{\bf k}\mid \geq  \dim
\mathrm{Supp}_{++}M,$ then we have  $\triangle ^{\mathrm{\bf
k}}P_M(\mathrm{\bf n}) = E(M;\mathrm{\bf k}).$  Hence by Theorem
\ref{th3.31}, we obtain the following  result.

 \vskip 0.2cm
\begin{theorem}\label{th3.20}  Let  \;$\mathbf{k}\in \mathbb{N}^d$ such that \;
 $\dim \mathrm{Supp}_{++}M \le\; \mid \mathrm{\bf k}\mid.$
Then  for any mixed  multiplicity system ${\bf x}$ of $M$ of the
type  $\mathrm{\bf k},$ we have $E(M;\mathrm{\bf k})=\chi({\bf x},
M) = \widetilde{e}({\bf x}, M).$
\end{theorem}

Recall that $S_{\bf 1}= S_{(1,\ldots,1)}.$ Now if  $S_{\bf
1}\nsubseteq \sqrt{\mathrm{Ann}_S M},$ then  $ \dim
\mathrm{Supp}_{++}M \geq 0.$ Set $ \dim \mathrm{Supp}_{++}M = s.$
Then we have $E(M;\mathrm{\bf k}) = e(M;\mathrm{\bf k})$ for each
$ \mathrm{\bf k}\in \mathbb{N}^d$ with $\mid \mathrm{\bf k}\mid
=s.$ Hence from  Theorem \ref{th3.20}  we immediately get the
following strong  result.
 \vskip 0.2cm

\begin{theorem}\label{th3.13} Assume that $S_{\bf 1}\nsubseteq \sqrt{\mathrm{Ann}_S M}.$
 Set $
\dim \mathrm{Supp}_{++}M = s.$  Let ${\bf x}$ be  a  mixed
multiplicity system of $M$ of the type  $ \mathrm{\bf k}\in
\mathbb{N}^d$ with $\mid \mathrm{\bf k}\mid =s.$ Then we have
$$e(M;\mathrm{\bf k}) = \chi({\bf x}, M)= \widetilde{e}({\bf x},
M).$$
\end{theorem}

For any $\mathrm{\bf k}\in \mathbb{N}^d$ with $\mid {\bf k} \mid =
\dim\mathrm{Supp}_{++}M,$ there exists a  mixed  multiplicity
system ${\bf x} = x_1,\ldots,x_s$ of $M$ of the type
 $\mathrm{\bf k}$ such that   for any  $ 0 \le i \le s,$
$$\dim\mathrm{Supp}_{++}\bigg[\dfrac{M}{(x_1,\ldots,
x_i)M}\bigg]\leq \dim\mathrm{Supp}_{++}M - i$$ by Remark \ref{re
3.0}. By Theorem \ref{th3.13}, we have $e(M;\mathrm{\bf k}) =
\widetilde{e}({\bf x}, M).$ Set ${\bf x}' = x_2,\ldots,x_s.$ And
assume that $x_1 \in S_i$.  Since $\widetilde{e}({\bf x}, M)=
\widetilde{e}({\bf x}', M/x_1M)- \widetilde{e}({\bf x}',
0_M:x_1),$  it follows that $ \widetilde{e}({\bf x}', M/x_1M) \ge
e(M;\mathrm{\bf k}).$  Note that   ${\bf x}'$ is a  mixed
multiplicity system  of the type $\mathrm{\bf k}- \mathrm{\bf
e}_i$ of both ${M/x_1M}$ and $0_M : x_1.$ Now assume that
$e(M;\mathrm{\bf k}) >  0.$  Then we have $ \widetilde{e}({\bf
x}', M/x_1M) > 0.$ Hence since $
\dim\mathrm{Supp}_{++}\big(\frac{M}{x_1M}\big) \le s-1,$ we get $
\dim\mathrm{Supp}_{++}\big(\frac{M}{x_1M}\big)= s-1 $ by Remark
\ref{re 3.3}. So in this case, we have
$$\dim\mathrm{Supp}_{++}\bigg(\frac{M}{x_1M}\bigg) =  \dim\mathrm{Supp}_{++}{M}-1.$$
Hence $\dim\mathrm{Supp}_{++}(0_M:x_1) \leq
\dim\mathrm{Supp}_{++}{M}-1$ by Remark \ref{re2.2}(ii). Remember that  by Theorem
\ref{th3.13}, we have $e(M/x_1M;\mathrm{\bf k}-{\bf e}_i) =
\widetilde{e}({\bf x}', M/x_1M).$ Therefore $e(M/x_1M;\mathrm{\bf
k}-{\bf e}_i) > 0.$ Consequently, by induction we easily get
$$\dim\mathrm{Supp}_{++}\big[\frac{M}{(x_1,\ldots,x_i)M}\big] =
\dim\mathrm{Supp}_{++}{M}-i\ge
\dim\mathrm{Supp}_{++}\frac{(x_1,\ldots,x_{i-1})M: x_i}
{(x_1,\ldots,x_{i-1})M}$$ for all $1 \le i \le s.$ Hence
$\dim\mathrm{Supp}_{++}\big(\frac{M}{{\bf x}M}\big) = 0.$  Denote
by $\mathrm{\bf h}_i= ({h_{i}}_1,\ldots,{h_{i}}_d)$ the type of
the subsequence $ x_1,\ldots,x_i$ (i.e., this sequence consists of
${h_{i}}_1$ elements of $S_1,\ldots,{h_{i}}_d$ elements of $S_d$)
of ${\bf x}$ for each $1\leq i \leq s.$ Then it is easily seen
that $x_{i+1},\ldots,x_s$ is a mixed  multiplicity system of
$\frac{(x_1,\ldots,x_{i-1})M: x_i}{(x_1,\ldots,x_{i-1})M}$ of the
type  $\mathrm{\bf k}- \mathrm{\bf h}_i$ and
$\dim\mathrm{Supp}_{++}\frac{(x_1,\ldots,x_{i-1})M:
x_i}{(x_1,\ldots,x_{i-1})M} \le \mid\mathrm{\bf k}-\mathrm{\bf
h}_i\mid$ for all $1\leq i \leq s.$  Hence
$$ \widetilde{e}\bigg(x_{i+1},\ldots,x_s, \frac{(x_1,\ldots,x_{i-1})M: x_i}{(x_1,
\ldots,x_{i-1})M}\bigg)= E\bigg(\frac{(x_1,\ldots,x_{i-1})M: x_i}{(x_1,\ldots,x_{i-1})M};
 {\bf k}-{\bf h}_i\bigg)$$ for all $1 \le i \le s$ by Theorem \ref{th3.20}.
 Note that
$$\widetilde{e}({\bf x},M) = \widetilde{e}\bigg(\emptyset, \frac{M}{{\bf x}M}\bigg)-
\sum_{i=1}^s\widetilde{e}\bigg(x_{i+1},\ldots,x_s,
\frac{(x_1,\ldots,x_{i-1})M: x_i}{(x_1,\ldots,x_{i-1})M}\bigg)$$
by Definition \ref{de3.4}  and $e(M;\mathrm{\bf k})=
\widetilde{e}({\bf x}, M)$ by Theorem \ref{th3.13}. Therefore
$$e(M;\mathrm{\bf k})= \widetilde{e}\bigg(\emptyset, \frac{M}{{\bf x}M}\bigg)-
 \sum_{i=1}^sE\bigg(\frac{(x_1,\ldots,x_{i-1})M: x_i}{(x_1,\ldots,x_{i-1})M};
  {\bf k}-{\bf h}_i\bigg).$$  Remember
that $\widetilde{e}\big(\emptyset, \frac{M}{{\bf x}M}\big) =
\ell_A\big[\big(\frac{M}{{\bf x}M}\big)_{\mathrm{\bf n}}\big]$ for
all large $\mathrm{\bf n}.$ The above facts yield:

 \vskip 0.2cm
\begin{corollary}\label{co3.17}
   Let ${\bf x} = x_1, \ldots, x_s$ be  a  mixed  multiplicity system  of $M$ of the
   type $\mathrm{\bf k}$  with $\mid {\bf k} \mid =
   \dim\mathrm{Supp}_{++}M$  and  $\dim\mathrm{Supp}_{++}\bigg[\dfrac{M}{(x_1,\ldots,
x_i)M}\bigg]\leq \dim\mathrm{Supp}_{++}M - i$  for any  $ 0 \le i
\le s.$
   Denote by $\mathrm{\bf h}_i= ({h_{i}}_1,\ldots,{h_{i}}_d)$ the type of the subsequence $ x_1,\ldots,x_i$ of ${\bf x}$ for each $1\leq i \leq s.$
 Assume that
$e(M;\mathrm{\bf k}) \ne 0.$
  Then the following statements hold.
\begin{itemize}
\item[$\mathrm{(i)}$]  For each $1\leq i \leq s,$ we have
$\dim\mathrm{Supp}_{++}\big[\frac{M}{(x_1, \ldots, x_i)M}\big] =
\dim\mathrm{Supp}_{++}{M}- i.$ \item[$\mathrm{(ii)}$]
$e(M;\mathrm{\bf k}) = \ell_A\big[\big(\frac{M}{{\bf
x}M}\big)_{\mathrm{\bf n}}\big] -
\sum_{i=1}^sE\bigg(\frac{(x_1,\ldots,x_{i-1})M:
x_i}{(x_1,\ldots,x_{i-1})M}; {\bf k}-{\bf h}_i\bigg)$ for all
large $\mathrm{\bf n}.$

\end{itemize}
\end{corollary}

Now assume that  ${\bf x}$ is  a  mixed  multiplicity system of
$M$ and  $a \in {\bf x}$ is an $S_{++}$-filter-regular element
with respect to  $M.$  Set ${\bf x}' = {\bf  x} \backslash \{a\}.$
Since  $a $ is an $S_{++}$-filter-regular element, we get
$\dim\mathrm{Supp}_{++}(0_M:a) < 0$ by Remark \ref{re2.2}(ii).
Consequently $\widetilde{e}({\bf x}', 0_M:a) = 0$ by Remark
\ref{re 3.3}. Therefore
$$\widetilde{e}({\bf x}, M) = \widetilde{e}({\bf x}', M/aM) - \widetilde{e}({\bf x}', 0_M:a)= \widetilde{e}({\bf x}', M/aM).$$
From this it follows that $\chi({\bf x}, M) = \chi({\bf x}', M/aM)$
by Proposition \ref {th3.5}.

We get the following corollary.

 \vskip 0.2cm

\begin{corollary}\label{co3.15}  Let  ${\bf x}$ be  a  mixed  multiplicity system of $M.$
 Let $a \in {\bf x}$ be an $S_{++}$-filter-regular element  with respect to  $M.$  Set ${\bf x}' = {\bf  x} \backslash \{a\}.$
Then
\begin{itemize}
\item[$\mathrm{(i)}$] $\chi({\bf x}, M) = \chi({\bf x}', M/aM).$
\item[$\mathrm{(ii)}$] $\widetilde{e}({\bf x}, M)= \widetilde{e}({\bf x}', M/aM).$
   \end{itemize}
   \end{corollary}

Let ${\bf x} = x_1, \ldots, x_s$ be an  $S_{++}$-filter-regular sequence
 of $M$ of the type $\mathrm{\bf k}$  with $\mid {\bf k} \mid \geq \dim\mathrm{Supp}_{++}M.$
  Then  ${\bf x}$ is a  mixed  multiplicity system of $M$
by Remark \ref{re 3.0}. And we have
 $$\widetilde{e}({\bf x}, M) = \widetilde{e}(x_{i+1},\ldots,x_s, M/(x_1,\ldots,x_i)M)$$
 for each $1 \le i \le s$ by Corollary \ref{co3.15}(ii).
Consequently, $$\widetilde{e}({\bf x}, M) =
\widetilde{e}\big(\emptyset, \frac{M}{{\bf x}M}\big) =
\ell_A\bigg[\big(\frac{M}{{\bf x}M}\big)_{\mathrm{\bf n}}\bigg]$$
for $\mathrm{\bf n} \gg \bf 0.$
  Note that $\ell_A\big[\big(\frac{M}{{\bf x}M}\big)_{\mathrm{\bf n}}\big] \ne 0$ for
  $\mathrm{\bf n} \gg \bf 0$    if and only if
  $\dim\mathrm{Supp}_{++}\big(\frac{M}{{\bf x}M}\big)= 0.$ So in this case,
$\mid \mathrm{\bf k}\mid = \dim \mathrm{Supp}_{++}M$ by Remark \ref{re 3.0}.

Hence  we have the following result.
 \vskip 0.2cm
\enlargethispage{1cm}

\begin{corollary} [see {\cite [Theorem 3.4]{Vi6}}] \label{co3.12}  Let ${\bf k}\in \mathbb{N}^d$ with $\mid {\bf k} \mid \ge \dim\mathrm{Supp}_{++}M.$  Assume that ${\bf x}$ is an  $S_{++}$-filter-regular sequence  with respect to  $M$ of the type ${\bf k}.$
Then we have
$$E(M;\mathrm{\bf k}) = \chi({\bf x}, M)= \widetilde{e}({\bf x}, M)=
 \ell_A\bigg[\bigg(\frac{M}{{\bf x}M}\bigg)_{\mathrm{\bf n}}\bigg]$$ for all large $\mathrm{\bf n}.$
And $E(M;\mathrm{\bf k})\ne 0 $\; if and only if
$\dim\mathrm{Supp}_{++}\big(\frac{M}{{\bf x}M}\big)= 0.$ In this
case, $\mid {\bf k} \mid = \dim\mathrm{Supp}_{++}M$ and
$E(M;\mathrm{\bf k}) = e(M;\mathrm{\bf k}).$
\end{corollary}
Let  $\Lambda$ be the set of all homogeneous prime ideals $P$ of  $S$ such that
$$P \in \mathrm{Supp}_{++}M \;\mathrm{ and}\; \dim \mathrm{Proj}\;(S/P) =
\dim \mathrm{Supp}_{++}M .$$   Then by \cite [Theorem 3.1]{VT1},
we have $$e(M;\mathrm{\bf k})= \sum_{P \in
\Lambda}\ell(M_P)e(S/P;\mathrm{\bf k}).$$ Recall that if $0 = M_0
\subseteq M_1 \subseteq M_2 \subseteq\cdots \subseteq M_u=M$ is a
prime  filtration of $M,$ then for any $P \in \Lambda,$ there
exists $0 \le i \le u-1$ such that  $S/P \cong M_{i+1}/M_i$ by the
proof of \cite [Theorem 3.1]{VT1}. Therefore, if ${\bf x}$ is  a
mixed  multiplicity system of $M,$ then ${\bf x}$ is also  a mixed
multiplicity system of $S/P$ for any $P \in \Lambda$ by Lemma
\ref{lem 3.30}.

 Hence by Theorem \ref{th3.13}, we immediately obtain the following corollary.

 \vskip 0.2cm
\begin{corollary}\label{co3.18}  Assume that  $S_{\bf 1}$ is not
contained in $ \sqrt{\mathrm{Ann}_S M}$. Denote by  $\Lambda$ the
set of all homogeneous prime ideals $P$ of $S$ such that $P \in
\mathrm{Supp}_{++}M$ and $\dim \mathrm{Proj}(S/P) = \dim
\mathrm{Supp}_{++}M .$  Set $ \dim \mathrm{Supp}_{++}M = s.$
Assume that ${\bf x}$ is  a  mixed  multiplicity system of $M$ of
the type  $\mathrm{\bf k}$  with $\mid \mathrm{\bf k}\mid \;=s.$
Then
\begin{itemize}
\item[$\mathrm{(i)}$] $\chi({\bf x}, M)= \sum_{P \in \Lambda}\ell(M_P)\chi({\bf x}, S/P).$
\item[$\mathrm{(ii)}$] $\widetilde{e}({\bf x}, M)= \sum_{P \in \Lambda}\ell(M_P)\widetilde{e}({\bf x}, S/P).$
\end{itemize}
\end{corollary}

\section{Applications to  Mixed Multiplicities of Ideals}

In this section,  we will give some applications of Section 3 to  mixed multiplicities of modules over local rings with respect to ideals.

 \vskip 0.2cm
Let
 $(R, \frak n)$  be  a  Noetherian   local ring with  maximal ideal $\frak{n}$  and infinite residue field $ R/\mathfrak{n}.$ Let $N$ be a finitely generated $R$-module.  Let $J, I_1,\ldots, I_d$ be ideals of $R$ with $J$ being  $\frak n$-primary   and $I_1\cdots I_d \nsubseteq \sqrt{\mathrm{Ann}_R{N}}.$ Recall that  $\mathrm{\bf k}= (k_1,\ldots,k_d) \in \mathbb{N}^d;$ $\mathrm{\bf n}= (n_1,\ldots,n_d)\in \mathbb{N}^d;$ $\mathrm{\bf I}= I_1,\ldots,I_d;$ $\mathrm{\bf I}^{[\mathrm{\bf k}]}= I_1^{[k_1]},\ldots,I_d^{[k_d]}$ and
 $\mathbb{I}^{\mathrm{\bf n}}= I_1^{n_1}\cdots I_d^{n_d}.$
  We get an $\mathbb{N}^{(d+1)}$-graded algebra and an $\mathbb{N}^{(d+1)}$-graded module:
$$ T  =\bigoplus_{n \ge 0,\; \mathrm{\bf n}\ge {\bf 0}}\dfrac{J^n\mathbb{I}^{\mathrm{\bf n}}}{J^{n+1}\mathbb{I}^{\mathrm{\bf n}}}\; \mathrm{and}\; {\cal N}  =\bigoplus_{n \ge 0,\; \mathrm{\bf n}\ge {\bf 0}}\dfrac{J^n\mathbb{I}^{\mathrm{\bf n}}N}{J^{n+1}\mathbb{I}^{\mathrm{\bf n}}N}.$$
Then  $T$ is a finitely generated standard
$\mathbb{N}^{(d+1)}$-graded algebra over an  Artinian local ring
$R/J$ and ${\cal N}$ is a finitely generated standard
$\mathbb{N}^{(d+1)}$-graded $T$-module. The mixed multiplicity of
${\cal N}$ of the type  $(k_0,\mathrm{\bf k})$ is denoted by
$e\big(J^{[k_0+1]},\mathrm{\bf I}^{[\mathrm{\bf k}]};N\big)$,
i.e.,
$$e\big(J^{[k_0+1]},\mathrm{\bf I}^{[\mathrm{\bf k}]};N\big) := e({\cal N}; k_0,\mathbf{k})$$  and which is called the {\it
mixed multiplicity of  $N$ with respect to ideals $J,\mathrm{\bf I}$
 of the type $(k_0+1,\mathrm{\bf k})$} (see \cite{HHRT,  MV, Ve}).
 Set $I=JI_1\cdots I_d;$ $I_0 = J;$ $T_i = I_i/JI_i$ for all
  $0 \le i \le d.$

 \vskip 0.2cm
\begin{remark}\label{de4.0}  On the
one hand $\deg P_{\cal N}(\mathrm{\bf n})= \dim
\mathrm{Supp}_{++}{\cal N}$ by \cite[Theorem 4.1]{HHRT}, and on
the other hand $\deg P_{\cal N}(\mathrm{\bf n})= \dim
\dfrac{N}{0_N: I^\infty}-1$ by \cite[Proposition 3.1(i)] {Vi}(see
\cite{MV}). Hence $ \dim \mathrm{Supp}_{++}{\cal N} = \dim
\dfrac{N}{0_N: I^\infty}-1.$
\end{remark}

Defining  mixed multiplicity systems of $\cal N$ is the reason for
using the following Rees superficial sequences.

 \vskip 0.2cm
\begin{definition}\label{de4.1}  An element $a \in R$ is called a {\it  Rees superficial element} of $N$ with respect to $\mathrm{\bf I}$ if there exists $i \in \{1, \ldots, d\}$ such that $a \in I_i$ and
 $aN \bigcap \mathbb{I}^{\mathrm{\bf n}}I_iN = a\mathbb{I}^{\mathrm{\bf n}}N$ for all
 $ \mathrm{\bf n}\gg \bf 0.$
Let $x_1, \ldots, x_t$ be a sequence in $R$.
 Then
$x_1, \ldots, x_t$ is called a {\it Rees superficial sequence}  of
$N$ with respect to $\mathrm{\bf I}$  if $x_{j + 1}$ is a Rees
superficial element  of $N/(x_1, \ldots, x_{j})N$  with respect to
$\mathrm{\bf I}$ for all $j = 0, 1, \ldots, t - 1.$ A Rees
superficial sequence of $N$  consisting of $k_1$ elements of
$I_1,\ldots,k_d$ elements of $I_d$ is called a Rees superficial
sequence of $N$ of {\it the type}  $\mathrm{\bf k}= (k_1
,\ldots,k_d).$
\end{definition}
 \vskip 0.2cm
\begin{remark}\label{re4.2} By \cite[Lemma 2.2]{MV} which is a generalized
version of \cite[Lemma 1.2]{Re}, for any set of ideals
$I_1,\ldots,I_d$ and each $\mathrm{\bf k}=(k_1,\ldots,k_d) \in
\mathbb{N}^{d},$ there exists a Rees superficial sequence of $N$
of the type  $\mathrm{\bf k}.$

Let ${\bf x}$ be a Rees superficial sequence of $N$ with respect
to ideals $J,\mathrm{\bf I}$ and let ${\bf x}^*$ be the image of
${\bf x}$ in $\bigcup_{i = 0}^dT_i.$ It can be verified (or see
\cite [Remark 4.1]{Vi6}) that $$\big({\cal N}/{\bf x}^*{\cal
N}\big)_{(m, \;\mathrm{\bf m})}\cong \bigg[\bigoplus_{n \ge 0,\;
\mathrm{\bf n}\ge {\bf 0}}\dfrac{J^n\mathbb{I}^{\mathrm{\bf
n}}(N/{\bf x}N)}{J^{n+1}\mathbb{I}^{\mathrm{\bf n}}(N/{\bf
x}N)}\bigg]_{(m, \;\mathrm{\bf m})}\eqno(4.1)$$ for all $m \gg
0;\; \mathrm{\bf m}\gg {\bf 0}.$ Hence by Remark \ref {de4.0}, we
have $$ \dim \mathrm{Supp}_{++}\frac{\cal N}{{\bf x}^*\cal N} =
\dim \dfrac{N}{{\bf x}N: I^\infty}-1.$$ So ${\bf x}^*$ is a  mixed
multiplicity system of $\cal N$ if and only if $\dim
\dfrac{N}{{\bf x}N: I^\infty}\le 1.$
\end{remark}

These comments lead us to define the following system.

 \vskip 0.2cm
\begin{definition}\label{de4.2} Let ${\bf x}$ be a Rees superficial sequence of $N$ with respect to ideals $J,\mathrm{\bf I}$ of the type $(k_0,\mathrm{\bf k}) \in \mathbb{N}^{d+1}.$ Then ${\bf x}$ is called a {\it mixed  multiplicity system of $N$ with respect to ideals $J,\mathrm{\bf I}$ of the type $(k_0,\mathrm{\bf k})$} if
$\dim \dfrac{N}{{\bf x}N: I^\infty}\le 1.$
\end{definition}

Under our point of view in this paper, both the above conditions of mixed  multiplicity systems  of $N$ are necessary to characterize mixed multiplicities of ideals in terms of the Euler-Poincare characteristic and
the mixed multiplicity symbol of $\cal N$ with respect to a mixed  multiplicity system of $\cal N.$

\vskip 0.2cm
\begin{remark}\label{re4.4}\rm  Remark \ref {de4.0} proves that if ${\bf x}$ is a Rees superficial sequence of $N$ with respect to $J,\mathrm{\bf I}$ and ${\bf x}^*$ is the image of ${\bf x}$ in $\bigcup_{i = 0}^dT_i,$ then ${\bf x}^*$ is a mixed  multiplicity system of ${\cal N}$ if and only if ${\bf x}$ is  a  mixed  multiplicity system of $N$ with respect to  $J,\mathrm{\bf I}.$
  \end{remark}
In  order to describe mixed multiplicity formulas of ideals we need to extend the notion of mixed multiplicities.

\begin{definition}\label{de4.5} Let   $k_0 \in \mathbb{N}$ and
$\mathrm{\bf k} \in \mathbb{N}^d$  such that
 $k_0 +\mid \mathrm{\bf k}\mid \; \ge \dim \dfrac{N}{0_N: I^\infty}-1.$ We assign
$$E(J^{[k_0+1]},\mathrm{\bf I}^{[\mathrm{\bf k}]}; N ) =
\begin{cases} e(J^{[k_0+1]},\mathrm{\bf I}^{[\mathrm{\bf k}]}; N )\; \text{ if } k_0 +\mid \mathrm{\bf k}\mid \; = \dim \dfrac{N}{0_N: I^\infty}-1, \\
0 \quad \quad \quad \quad \quad \quad\quad \text{ if } k_0+\mid \mathrm{\bf k}\mid  \;> \dim \dfrac{N}{0_N: I^\infty}-1.
\end{cases}$$
\end{definition}

 Put $I_0 = J.$
Next, assume that $a \in I_i$ is a Rees superficial element of $N$
with respect to $J, \mathrm{\bf I}$ and $a^*$ is the image of $a$
in $T_i.$ Then we have
\begin{align*}
&(J^{n+1}\mathbb{I}^\mathrm{\bf n}I_iN:a) \bigcap J^n\mathbb{I}^\mathrm{\bf n}N \\
&= \big((J^{n+1}\mathbb{I}^\mathrm{\bf n}I_iN\bigcap aN):a\big)\bigcap J^n\mathbb{I}^\mathrm{\bf n}N \\
&= \big(aJ^{n+1}\mathbb{I}^\mathrm{\bf n}N:a\big) \bigcap J^n\mathbb{I}^\mathrm{\bf n}N \\
&= \big(J^{n+1}\mathbb{I}^\mathrm{\bf n}N + (0_N : a)\big) \bigcap J^n\mathbb{I}^\mathrm{\bf n}N\\
&= J^{n+1}\mathbb{I}^\mathrm{\bf n}N  + (0_N : a) \bigcap J^{n}\mathbb{I}^\mathrm{\bf n}N
\end{align*}
 for all  $ n \gg 0;\;\mathrm{\bf n} \gg {\bf 0}.$ Consequently
$$(J^{n+1}\mathbb{I}^\mathrm{\bf n}I_iN:a) \bigcap J^n\mathbb{I}^\mathrm{\bf n}N= J^{n+1}\mathbb{I}^\mathrm{\bf n}N  + (0_N : a) \bigcap J^{n}\mathbb{I}^\mathrm{\bf n}N \eqno(4.2)$$
 for all  $n \gg 0;\;\mathrm{\bf n} \gg {\bf 0}.$
From (4.2) it follows that
\begin{align*}
(0_{\cal N}: a^*)_{(n,\;\mathrm{\bf n})} &= \dfrac{(J^{n+1}\mathbb{I}^\mathrm{\bf n}I_iN:a) \bigcap J^n\mathbb{I}^\mathrm{\bf n}N}{J^{n+1}\mathbb{I}^\mathrm{\bf n}N}\\
&= \dfrac{J^{n+1}\mathbb{I}^\mathrm{\bf n}N  + (0_N : a) \bigcap J^{n}\mathbb{I}^\mathrm{\bf n}N}{J^{n+1}\mathbb{I}^\mathrm{\bf n}N}\\
&= \dfrac{(0_N : a) \bigcap J^{n}\mathbb{I}^\mathrm{\bf n}N}{(0_N : a) \bigcap J^{n+1}\mathbb{I}^\mathrm{\bf n}N}
\end{align*} for all  $n \gg 0;\;\mathrm{\bf n} \gg {\bf 0}.$
 Remember that  $I=JI_1\cdots I_d.$
By Artin-Rees lemma, there exists $u \gg 0$ such that
$$(0_{\cal N}: a^*)_{(n+u,\; {\bf n}+u{\bf 1})} = \dfrac{[(0_N : a)\bigcap I^uN ]J^{n}\mathbb{I}^\mathrm{\bf n}}{[(0_N : a)\bigcap I^uN ] J^{n+1}\mathbb{I}^\mathrm{\bf n}}\eqno(4.3)$$
for all  $ n \ge 0;\;\mathrm{\bf n} \ge {\bf 0}.$ Fix the $u$ and
set $W = (0_N : a)\bigcap I^uN ;\; U = (0_N : a)/W.$ Since $W:
I^\infty = 0_N:aI^\infty \supseteq 0_N:a,$ it follows that
$U/0_U:I^\infty = 0.$ So $\dim U/0_U:I^\infty < 0.$ Hence since
the exact sequence $0\longrightarrow W \longrightarrow
(0_N:a)\longrightarrow U = (0_N:a)/W\longrightarrow 0,$ we get
that the mixed multiplicities of $(0_N:a)$ and the mixed
multiplicities of $W$ with respect to ideals $J,{\bf I}$ are the
same by \cite[Corollary 3.9 (ii)]{VT1}.
 So  $$E(J^{[k_0+1]},\mathrm{\bf I}^{[\mathrm{\bf k}]}; W ) = E(J^{[k_0+1]},\mathrm{\bf I}^{[\mathrm{\bf k}]}; 0_N:a).$$
On the other hand by (4.3), we have $E(0_{\cal N}:a^*; k_0,{\bf
k})= E(J^{[k_0+1]},\mathrm{\bf I}^{[\mathrm{\bf k}]}; W ).$ Hence
$E(0_{\cal N}:a^*; k_0,{\bf k})= E(J^{[k_0+1]},\mathrm{\bf
I}^{[\mathrm{\bf k}]}; 0_N:a).$ By (4.1), we get
$$E({\cal N}/a^*{\cal N}; k_0,{\bf k})= E(J^{[k_0+1]},\mathrm{\bf I}^{[\mathrm{\bf k}]}; N/aN).$$

The above facts yield:

\vskip 0.2cm
\noindent
\begin{remark}\label{re4.6} We have $E({\cal N}; k_0,\mathrm{\bf k}) = E(J^{[k_0+1]},
\mathrm{\bf I}^{[\mathrm{\bf k}]}; N ).$ And if $a \in I_i$ is a
Rees superficial element of $N$ with respect to $J, \mathrm{\bf
I}$ and $a^*$ is the image of $a$ in $T_i,$ then
\begin{align*}
E(0_{\cal N}:a^*; k_0,{\bf k})&= E(J^{[k_0+1]},\mathrm{\bf
I}^{[\mathrm{\bf k}]}; 0_N:a);\\E({\cal N}/a^*{\cal N}; k_0,{\bf
k})&= E(J^{[k_0+1]},\mathrm{\bf I}^{[\mathrm{\bf k}]};
N/aN).\end{align*} Moreover, if $\dim \dfrac{N}{0_N: I^\infty} =
1,$ then by \cite[Proposition 3.2]{Vi} it can easily be seen that
$E(J^{[1]},\mathrm{\bf I}^{[\mathrm{\bf 0}]}; N) = e(J
;\frac{N}{0_N: I^{\infty}})$.

\end{remark}
Now,  let ${\bf x}$ be  a  mixed  multiplicity system of $N$ with
respect to  $J,\mathrm{\bf I}$ of the type $(k_0, {\bf k})$ and
let ${\bf x}^*$ be the image of ${\bf x}$ in $\bigcup_{i =
0}^dT_i.$ Then ${\bf x}^*$ is a  mixed  multiplicity system of
$\cal N$ of the type $(k_0, {\bf k})$ by Remark \ref{re4.4}. Hence
$\widetilde{e}({\bf x}^*, {\cal N}) = \chi({\bf x}^*, {\cal N})$
by Proposition \ref {th3.5}. And
 $\widetilde{e}({\bf x}^*, {\cal N}) = \chi({\bf x}^*, {\cal N})=
 E({\cal N}; k_0,\mathrm{\bf k})$ if $k_0 + \mid{\bf k}\mid \ge
\dim \mathrm{Supp}_{++}{\cal N}$
  by Theorem \ref{th3.20}. On the one hand
$E({\cal N}; k_0,\mathrm{\bf k}) = E(J^{[k_0+1]},\mathrm{\bf I}^{[\mathrm{\bf k}]}; N )$ by Remark \ref{re4.6}. On the other hand
$ \dim \mathrm{Supp}_{++}{\cal N} = \dim \dfrac{N}{0_N: I^\infty}-1$ by
 Remark \ref{de4.0}. Therefore,
 $\widetilde{e}({\bf x}^*, {\cal N}) = \chi({\bf x}^*, {\cal N})=
 E(J^{[k_0+1]},\mathrm{\bf I}^{[\mathrm{\bf k}]}; N )$
 if
$k_0 + \mid{\bf k}\mid \ge \dim \dfrac{N}{0_N: I^\infty}-1.$
   Consequently, we have:

 \vskip 0.2cm
\begin{corollary}\label{co4.7} Let ${\bf x}$ be  a  mixed  multiplicity system of
$N$ with respect to ideals $J,\mathrm{\bf I}$ of the type $(k_0,\mathrm{\bf k})$ and let ${\bf x}^*$ be the image of ${\bf x}$ in $\bigcup_{i = 0}^dT_i.$ Then the following statements hold.
\begin{itemize}
\item[$\mathrm{(i)}$] ${\bf x}^*$ is a  mixed  multiplicity system
of $\cal N$ of the type $(k_0,\mathrm{\bf k}).$
\item[$\mathrm{(ii)}$] $\chi({\bf x}^*, {\cal N})=
\widetilde{e}({\bf x}^*, {\cal N}).$ \item[$\mathrm{(iii)}$] If
$k_0 + \mid{\bf k}\mid \; \ge \dim \dfrac{N}{0_N: I^\infty}-1,$
then $E(J^{[k_0+1]},\mathrm{\bf I}^{[\mathrm{\bf k}]}; N )=
\chi({\bf x}^*, {\cal N})= \widetilde{e}({\bf x}^*, {\cal N}) .$
\end{itemize}
\end{corollary}
In the case that $JI_1\cdots I_d$  is not contained in $ \sqrt{\mathrm{Ann}_R{N}},$ by Corollary \ref{co4.7} we immediately get the main result of this section.

 \vskip 0.4cm

\begin{theorem}\label{th4.8}
  Let  $J, I_1,\ldots, I_d$  be ideals of $R$ with   $J$ being  $\frak n$-primary.
   Assume  that $I=JI_1\cdots I_d$  is not contained in $ \sqrt{\mathrm{Ann}_R{N}}$ and $(k_0, \mathrm{\bf k}) \in \mathbb{N}^{d+1}$ such that $k_0+ \mid\mathrm{\bf k}\mid \;= \dim \dfrac{N}{0_N: I^\infty}-1.$
  Let ${\bf x}$ be  a  mixed  multiplicity system of $N$ with respect to ideals $J,\mathrm{\bf I}$ of the type $(k_0,\mathrm{\bf k})$ and ${\bf x}^*$  the image of ${\bf x}$ in $\bigcup_{i = 0}^dT_i.$ Then we have
$$e(J^{[k_0+1]},\mathrm{\bf I}^{[\mathrm{\bf k}]}; N ) = \chi({\bf x}^*, {\cal N})=\widetilde{e}({\bf x}^*, {\cal N}).$$
\end{theorem}

Now by Corollary \ref{co3.17} and Remark \ref{re4.6}, we obtain the following corollary.
 \vskip 0.2cm
\begin{corollary}\label{co4.11} Assume  that $I=JI_1\cdots I_d$  is not contained in
$ \sqrt{\mathrm{Ann}_R{N}}$ and  $e(J^{[k_0+1]},\mathrm{\bf I}^{[\mathrm{\bf k}]}, N)\ne 0.$
 Let ${\bf x} = x_1, \ldots, x_s$ be a  mixed  multiplicity system of $N$ with
 respect to ideals $J,\mathrm{\bf I}$ of the type $(k_0,\mathrm{\bf
 k})$ such that for each $1\leq i \leq s,$ $$\dim \dfrac{N}{(x_1, \ldots, x_i)N: I^\infty} \le
 \dim \dfrac{N}{0_N: I^\infty} \;- i.$$
 Denote by $(m_i,\mathrm{\bf h}_i)= (m_i, {h_{i}}_1,\ldots,{h_{i}}_d)$ the type of the subsequence $ x_1,\ldots,x_i$ of ${\bf x}$ for each $1\leq i \leq s.$ And for each $1\leq i \leq s,$ set $$N_i = \frac{(x_1,\ldots,x_{i-1}){N}: x_i}{(x_1,\ldots,x_{i-1}){N}}.$$
Then the following statements hold.

\begin{itemize}
\item[$\mathrm{(i)}$]  $\dim \dfrac{N}{(x_1, \ldots, x_i)N: I^\infty}= \dim \dfrac{N}{0_N: I^\infty} \;- i$ for each $1\leq i \leq s.$
\item[$\mathrm{(ii)}$]
$e(J^{[k_0+1]},\mathrm{\bf I}^{[\mathrm{\bf k}]}; N) = e\big(J; \dfrac{N}{{\bf x}N:I^\infty}\big)
- \sum_{i=1}^sE\big(J^{[k_0-m_i+1]},\mathrm{\bf I}^{[\mathrm{\bf k}-\mathrm{\bf h}_i]}; N_i\big).$

\end{itemize}
\end{corollary}
\begin{proof}
 Since ${\bf x}$ is a  mixed  multiplicity system of $N$ with respect to ideals
  $J,\mathrm{\bf I}$ of the type $(k_0,\mathrm{\bf k})$ such that  $\dim \dfrac{N}{(x_1, \ldots, x_i)N: I^\infty} \le
\dim \dfrac{N}{0_N: I^\infty} \;- i$ for each
  $1\leq i \leq s,$
   ${\bf x}^*$
  is a  mixed  multiplicity system  of the type  $(k_0, \mathrm{\bf k})$  of $\cal N$
such that for each $1\leq i \leq s,$
$$\dim\mathrm{Supp}_{++}\big[\frac{\cal N}{(
x^*_1,\ldots,x^*_{i})\cal N}\big] \le \dim\mathrm{Supp}_{++}{\cal
N}- i$$
   by Remark \ref {de4.0} and
  Remark  \ref{re4.4}. Recall that  $e({\cal N}; k_0,\mathrm{\bf k}) = e(J^{[k_0+1]},
  \mathrm{\bf I}^{[\mathrm{\bf k}]}; N).$ Hence $e({\cal N}; k_0,\mathrm{\bf k}) \ne 0$ since  $e(J^{[k_0+1]},\mathrm{\bf I}^{[\mathrm{\bf k}]}; N)\ne 0.$
Consequently  by Corollary \ref{co3.17}(i), we have
$$\dim\mathrm{Supp}_{++}\big[\frac{\cal N}{( x^*_1,\ldots,x^*_{i})\cal N}\big]
 = \dim\mathrm{Supp}_{++}{\cal N}- i$$
for each $1\leq i \leq s.$ Therefore  $\dim
\dfrac{N}{(x_1,\ldots,x_{i})N: I^\infty}= \dim \dfrac{N}{0_N:
I^\infty} \;- i$ for each $1\leq i \leq s$  by Remark \ref{de4.0}.
We get (i). Now,  it is easily seen by  Remark \ref{re4.6} and
(4.1) that $e({\cal N}/{\bf x}^*{\cal N}; 0, {\bf 0}) = e(J
;\frac{N}{{\bf x}N: I^{\infty}})$  and for each $ 1 \le i \le s,$
\begin{align*}
&E\bigg(\frac{(x^*_1,\ldots,x^*_{i-1}){\cal N}: x^*_i}{(x^*_1,\ldots,x^*_{i-1}){\cal N}}; k_0-m_i,\mathrm{\bf k}-\mathrm{\bf h}_i\bigg)\\
&= E\bigg(J^{[k_0-m_i+1]},\mathrm{\bf I}^{[\mathrm{\bf k}-\mathrm{\bf h}_i]}; \frac{(x_1,\ldots,x_{i-1}){N}: x_i}{(x_1,\ldots,x_{i-1}){N}}\bigg).
\end{align*}
Consequently,  by Corollary \ref{co3.17}(ii)  we get (ii).
\end{proof}

In particular, let  ${\bf I}= I_1,\ldots,I_d$ be $\frak n$-primary
ideals. Then  $e(J^{[k_0+1]},\mathrm{\bf I}^{[\mathrm{\bf k}]},
N)\ne 0$ by \cite{Te}.  Moreover,  $\dim \dfrac{N}{{\bf x}N}=\dim
\dfrac{N}{{\bf x}N:I^\infty}$ and $e\big(J; \dfrac{N}{{\bf
x}N:I^\infty}\big) = e\big(J; \dfrac{N}{{\bf x}N}\big)$ by
\cite{Vi}. Hence since $\dim \dfrac{N}{{\bf x}N:I^\infty} \le 1,$
it follows that
 $ \dim \dfrac{N}{{\bf x}N}
\le \dim N - s .$ So ${\bf x}$ is a part of a parameter system for
$N.$

Then as an immediate consequence of  Corollary \ref{co4.11}, we
have the following result.
 \vskip 0.2cm
\begin{corollary}\label{co4.16} Let $J; {\bf I}$ be  $\frak n$-primary ideals and $\dim N > 0.$
 Let ${\bf x} = x_1, \ldots, x_s$ be a  mixed  multiplicity system of $N$ with respect to
 ideals $J,\mathrm{\bf I}$ of the type $(k_0,\mathrm{\bf k}).$
 Denote by $(m_i,\mathrm{\bf h}_i)= (m_i, {h_{i}}_1,\ldots,{h_{i}}_d)$
 the type of the subsequence $ x_1,\ldots,x_i$ of ${\bf x}$ for each $1\leq i \leq s.$ And put
 $N_i = \frac{(x_1,\ldots,x_{i-1}){N}: x_i}{(x_1,\ldots,x_{i-1}){N}}$ for each $1\leq i \leq s.$
Then we have the following statements.
\begin{itemize}
\item[$\mathrm{(i)}$] ${\bf x}$ is a part of a parameter system  for $N.$
\item[$\mathrm{(ii)}$]
$e(J^{[k_0+1]},\mathrm{\bf I}^{[\mathrm{\bf k}]}; N) = e\big(J; \dfrac{N}{{\bf x}N}\big)
- \sum_{i=1}^sE\big(J^{[k_0-m_i+1]},\mathrm{\bf I}^{[\mathrm{\bf k}-\mathrm{\bf h}_i]}; N_i\big).$
\end{itemize}
\end{corollary}

Using different sequences, one expressed
  mixed multiplicities of ideals in terms of the Hilbert-Samuel multiplicity. For instance: in the case of $\mathfrak{n}$-primary ideals, Risler-Teissier \cite{Te} in 1973 showed that each mixed multiplicity is the multiplicity of an ideal generated by a superficial sequence and  Rees \cite{Re} in 1984 proved that mixed multiplicities are multiplicities of ideals generated by joint reductions; for
the case of arbitrary ideals, Viet \cite{Vi} in 2000 characterized mixed multiplicities as the  Hilbert-Samuel multiplicity via (FC)-sequences.
\vskip 0.4cm
 \begin{definition}[\cite{Vi}]\label{de4.9}  Let $\mathrm{\bf I}= I_1,\ldots,I_d$ be ideals of $R.$ Set ${\frak I} = I_1\cdots I_d.$ An element $a \in R$ is called a {\it weak-$(FC)$- element of $N$ with respect to $\mathrm{\bf I}$}  if there exists $i \in \{ 1, \ldots, d\}$ such that $a \in I_i$ and
the following conditions are satisfied:
 \begin{itemize}
      \item[$\mathrm{(i)}$] $a$ is an $\frak I$-filter-regular element with respect to $N,$ i.e.,\;$0_N:a \subseteq 0_N: {\frak I}^{\infty}.$
    \item[$\mathrm{(ii)}$] $a$ is a Rees superficial element of $N$ with respect to $\mathrm{\bf I}.$
\end{itemize}
\noindent
\end{definition}
Note that \cite{Vi} defined weak-(FC)-sequences in the condition
$\frak I \nsubseteq \sqrt{\mathrm{Ann}_R{N}}$ (see e.g. \cite
{DMT, DV, MV,Vi1,Vi2, Vi4,VDT, VT, VT1}). In  Definition
\ref{de4.9}, we omitted  this condition.

We end this paper with the following result of mixed multiplicities of ideals.
   \vskip 0.2cm
\begin{corollary}\label{co4.10}  Let $k_0+ \mid {\bf k}\mid = \dim \dfrac{N}{0_N: I^\infty}-1$.  Let ${\bf x}$ be  a  weak-$(FC)$-sequence of $N$ with respect to $J, \mathrm{\bf I}$ of the type $(k_0,\mathrm{\bf k})$ and let ${\bf x}^*$ be the image of ${\bf x}$ in $\bigcup_{i = 0}^dT_i.$ Then
  \begin{itemize}
\item[$\mathrm{(i)}$]
    $\widetilde{e}({\bf x}^*, {\cal N}) = \chi({\bf x}^*, {\cal N})=e(J^{[k_0+1]},\mathrm{\bf I}^{[\mathrm{\bf k}]}; N )= E(J^{[1]},{\bf I}^{[{\bf 0}]}; \frac{N}{{\bf x}N}).$
\item[$\mathrm{(ii)}$] $e(J^{[k_0+1]},\mathrm{\bf I}^{[\mathrm{\bf
k}]}; N )\ne 0 $\; if and only if  $\dim\frac{N}{{\bf x}N:
I^{\infty}}= 1.$ In this case,  $$e(J^{[k_0+1]},\mathrm{\bf
I}^{[\mathrm{\bf k}]}; N )= e(J;\frac{N}{{\bf x}N:I^{\infty}}).$$
\end{itemize}
\end{corollary}

\begin{proof} Since ${\bf x}$ is a  weak-$(FC)$-sequence of $N$ with respect to $J, \mathrm{\bf I}$ of the type $(k_0,\mathrm{\bf k}),$
${\bf x}^*$ is a  $T_{++}$-filter-regular sequence  with respect
to  $\cal N$ of the type $(k_0,\mathrm{\bf k})$ by
\cite[Proposition 4.5]{VT1}. Since $k_0+ \mid {\bf k}\mid = \dim
\dfrac{N}{0_N: I^\infty}-1,$ $ \dim \mathrm{Supp}_{++}{\cal N} =
k_0+ \mid {\bf k}\mid$ by Remark \ref{de4.0}. So $E({\cal N};k_0,
\mathrm{\bf k}) = e({\cal N};k_0, \mathrm{\bf k}).$ Hence by
Corollary \ref{co3.12}, we have $$e({\cal N};k_0, \mathrm{\bf k})
= E({\cal N}/{\bf x}^*{\cal N};0, \mathrm{\bf 0}).$$  Remember
that
 $E({\cal N}/{\bf x}^*{\cal N}; 0, {\bf 0}) = E(J^{[1]}, {\bf I}^{[{\bf 0}]}; N/{\bf x}N)$ by Remark \ref{re4.6} and
$$e({\cal N};k_0, \mathrm{\bf k}) = e(J^{[k_0+1]},\mathrm{\bf I}^{[\mathrm{\bf k}]}; N ).$$ Consequently, by Theorem \ref{th4.8} we get (i) that
$$\widetilde{e}({\bf x}^*, {\cal N}) = \chi({\bf x}^*, {\cal N})=e(J^{[k_0+1]},\mathrm{\bf I}^{[\mathrm{\bf k}]}; N )= E(J^{[1]},{\bf I}^{[{\bf 0}]}; \frac{N}{{\bf x}N}).$$
So $e(J^{[k_0+1]},\mathrm{\bf I}^{[\mathrm{\bf k}]}; N )\ne 0 $\;
if and only if $E(J^{[1]}, {\bf I}^{[{\bf 0}]}; \frac{N}{{\bf
x}N}) \ne 0.$ This is equivalent to $\dim\frac{N}{{\bf x}N:
I^{\infty}}= 1$ by Corollary \ref{co3.12}. In this case,
 $E(J^{[1]}, {\bf I}^{[{\bf 0}]}; \frac{N}{{\bf x}N}) = e(J;\frac{N}{{\bf x}N: I^{\infty}})
 $ by Remark \ref{re4.6}.
Hence
 $e(J^{[k_0+1]},\mathrm{\bf I}^{[\mathrm{\bf k}]}; N )= e(J;\frac{N}{{\bf
 x}N:I^{\infty}})$ by (i).
 \end{proof}

Note  that Corollary \ref{co4.10}(ii) is also an immediate consequence of \cite[Theorem 3.4]{Vi}(see e.g. \cite{
DV, MV,Vi2,VDT,  VT}).

\end{document}